\documentclass[a4paper]{article}%

\usepackage{amsmath,amsthm,amssymb,fancyhdr,color,epic,graphicx,appendix}%

\newtheorem{theorem}{\bf Theorem}[section]

\newtheorem{lemma}[theorem]{\bf Lemma}
\newtheorem{proposition}[theorem]{\bf Proposition}
\newtheorem{corollary}[theorem]{\bf Corollary}
\renewenvironment{proof}[1][Proof]{\noindent\textbf{#1.} }{\ \rule{0.4em}{0.4em}}

\newenvironment{keywords}{\par\noindent\textbf{Keywords: }}{\hfill\par}

\theoremstyle{definition}

\newtheorem{remark}[theorem]{\bf Remark}

\definecolor{Be}{rgb} {0,0.08,0.45} 
\definecolor{darkgreen}{rgb}{0.00, 0.49, 0.00}
\definecolor{Mahog}{rgb}{0.6,  0.,   0}

\newcommand{\R}{I\!\!R}

\numberwithin{equation}{section}

\begin{document}

\title{Asymptotic analysis of a Schr\"odinger-Poisson system with
quantum wells and macroscopic nonlinearities in dimension 1}
\author{A. Faraj\thanks{IRMAR, UMR - CNRS 6625, Universit\'{e} de Rennes 1, 35042 Rennes Cedex, France.}}
\date{}
\maketitle

\begin{abstract}
We consider the stationary one dimensional Schr\"{o}dinger-Poisson system on a bounded interval with a background
potential describing a quantum well. Using a partition function which forces the particles
to remain in the quantum well, the limit $h\rightarrow0$ in the nonlinear
system leads to a uniquely solved nonlinear problem with concentrated particle density. It allows to conclude about the convergence of the solution.
\end{abstract}

\section{Introduction}\label{mod_limSP}

The quantum state of a gas of charged particles is described, in the mean
field approximation, by a nonlinear one-particle Schr\"{o}dinger equation
where the electrostatic repulsion is modeled by a nonlinear potential term
depending on the charge density through a Poisson equation. This class of
models is usually referred to as Schr\"{o}dinger-Poisson systems. In this work
we consider a stationary Schr\"{o}dinger-Poisson system in a bounded region of
$\mathbb{R}$, for which a background potential models a quantum
well, while the nonlinear potential extends on a wider scale. After
introducing a rescaling for which the small parameter $h>0$ represents an
inverse length scale, the support of the potential well squeezes
asymptotically to a single point in the limit $h\rightarrow0$. An equilibrium
state of a gas of charged particles confined in the quantum well will be
considered, while the nonlinear electrostatic potential created by such a
concentrated charge tends to a potential which is picewise linear and almost constant in the well.

Such a Schr\"{o}dinger-Poisson problem has recently been considered in
\cite{BNP1}, \cite{BNP2} and \cite{Ni2} in a more complex setting involving far from equilibrium steady states. This
one-dimensional analysis leads to a reduced model which happens to be very
efficient in the numerical simulation of the electronic transport through
semiconductor heterostructures, like resonant tunneling diodes, see \cite{BFN} and \cite{BNP3}. Nevertheless, in the present work the Schr\"odinger-Poisson system has a unique solution and the analysis shows that the unicity is asymptotically preserved. It is not the case in \cite{BNP1}, \cite{BNP2} and \cite{Ni2} where hysteresis phenomena are predicted (e.g. in \cite{JLPS} and
\cite{PrSj}).

For the sake of simplicity we shall use a low energy-filter in the definition
of the partition function $f$ (see equation (\ref{source}) below), that is the
quantum states with an energy larger than the threshold $\varepsilon_{S}$ are
not occupied. With such an assumption only the quantum states confined in the
well have an effect on the nonlinearity. This is an important point to get asymptotically the macroscopic quantities.

The semi-classical analysis of such a model was performed in \cite{FMN} in dimension $d=2$ and $3$ and it appears that in the limit, the potential vanishes almost everywhere and produces a non null spectral perturbation. This apparent contradiction is solved through a rescaling with parameter $h$ and the potential behaviour depends on the Green function of the Laplace operator. For the $1D$ problem, the nonlinear effect remains visible at the macroscopic scale in the limit $h\rightarrow0$, this provides a non trivial approximation of the solution of the Schr\"odinger-Poisson system. 

Like in \cite{BNP1}, \cite{BNP2}, \cite{FMN} and \cite{Ni2} the analysis will be a mixture
of nonlinear apriori estimates combined with accurate semiclassical and
spectral techniques (we refer to: \cite{DiSj}, \cite{He}, \cite{HeSj} and \cite{ReSi}) adapted for potentials with limited regularity. The outline of
this analysis is the following. We end this section by introducing the model and by stating our results.
In Section \ref{sec_limlin}, we give some asymptotics for the spectrum of the linear Hamiltonian. In Section \ref{sec_premres}, we present preliminary results for the Schr\"odinger-Poisson system. The limit of the Schr\"odinger-Poisson system is done in Section \ref{sec_limSPd=1}. The semi-classical analysis tools, necessary for the asymptotics, are given in Appendix \ref{sec_Agmdist} and \ref{sec_expdec}.

\subsection{The model}

Let $\Omega=(0,L)$ be an open bounded interval and $U$ a
non positive function in $C_{0}^{\infty}(\mathbb{R})$ supported in the
ball of radius one centered in the origin. 

For $x_{0}\in\Omega$, we define the potential with center $x_{0}$ and radius
of order $h>0$%
\[
\displaystyle U^{h}(x)=U\left(  \frac{x-x_{0}}{h}\right)  ,\quad x\in\Omega\,.
\]
Our analysis being concerned with the limit $h\rightarrow0$, we can choose,
without loss of generality, $h$ small enough so that the support of $U^{h}$ is
included in $\Omega$. In particular, defining with $\omega^{h}$
the support of $U^{h}$, we assume that $\omega
^{h}\subset\Omega$ for all values of $h$ below a suitable positive constant:
$h\leq h_{0}$.

Next we assign the function $f\in C^{\infty}(\mathbb{R})$, with a threshold at
$\varepsilon_{S}<0$ and fulfilling the conditions%
\begin{equation}
\displaystyle f(x)>0,\quad\forall x<\varepsilon_{S}\,, \label{condition 1}%
\end{equation}%
\begin{equation}
\displaystyle f(x)=0,\quad\forall x\geq\varepsilon_{S}\,, \label{condition 2}%
\end{equation}%
\begin{equation}
\displaystyle f^{\prime}(x)\leq0,\quad\forall x\in\mathbb{R\,,} \label{condDec}%
\end{equation}
and address, for $h\in\left(  0,h_{0}\right]  $, the following problem: find
$V^{h}$ solving the nonlinear Poisson equation%
\begin{equation}
\left\{
\begin{array}
[c]{l}%
\displaystyle-\frac{d^2}{dx^2} V^{h}=n[V^{h}]\quad\text{in}\ \Omega\\[1.65mm]
\displaystyle\left.  V^{h}\right\vert _{\partial\Omega}=0
\end{array}
\right.  \label{Poiss}%
\end{equation}
where the source term is%
\begin{equation}
\displaystyle n[V^{h}]=\sum_{i\geq 1}f(\varepsilon_{i}^{h})|\Psi_{i}^{h}|^{2}\,,
\label{source}%
\end{equation}
and $\left\{  \varepsilon_{i}^{h}\right\}  _{i\geq 1}$ are the
eigenvalues of the nonlinear Hamiltonian%
\begin{equation}
\displaystyle H^{h}=-h^{2}\frac{d^2}{dx^2}+U^{h}+V^{h}%
\end{equation}
numerated from $\inf\sigma(H^{h})$ counting multiplicities, while $\left\{
\Psi_{i}^{h}\right\}  _{i\geq 1}$ are the corresponding eigenvectors%
\begin{equation}
\left\{
\begin{array}
[c]{l}%
\displaystyle H^{h}\Psi_{i}^{h}=\varepsilon_{i}^{h}\Psi_{i}^{h},\quad\text{in}\ \Omega\,,\\[1.65mm]
\displaystyle \left.  \Psi_{i}^{h}\right\vert _{\partial\Omega}=0\,.
\end{array}
\right.  \label{Schrod}%
\end{equation}
The equations (\ref{Poiss}), (\ref{source}) and (\ref{Schrod}) define the
stationary Schr\"{o}dinger-Poisson system associated with the potential well
$U^{h}$ and the function $f$. In practical applications, where these equations
are used for the description of the charge distribution in electronic devices,
$n[V^{h}]$ describes the density of the charge careers of the system, while
$f$ is a response function which depends on the characteristics of the device
and has to be considered as a data item of the problem.

The stationary states form a set of
real normalized functions%
\begin{equation}
\displaystyle \operatorname{Im}\Psi_{i}^{h}=0;\left\Vert \Psi_{i}^{h}\right\Vert
_{L^{2}(\Omega)}=1\,, \label{condition eigenvectors 1}%
\end{equation}

The analysis of our Schr\"{o}dinger-Poisson system, will involve the operator
\begin{equation}
\displaystyle H_{0}=-\frac{d^2}{dx^2}+U;\quad D(H_{0})=H^{2}(\mathbb{R})\,, \label{H0}%
\end{equation}
whose point spectrum, $\sigma_{p}(H_{0})$, contains a finite number of points
embedded in $\left[  -\left\Vert U\right\Vert _{L^{\infty}},0\right)  $. In
particular, we make the following assumption
\begin{equation}
\displaystyle e_{1}:=\inf\sigma\left(  H_{0}\right)  <\varepsilon_{S}\ . \label{condition 4}%
\end{equation}
The hypothesis (\ref{condition 4}) -- which prevents the
solution to (\ref{Poiss}) - (\ref{Schrod}) to be trivial -- will be
extensively used in this work.

\subsection{Results}

In our one dimensional case (see \cite{Ni0}) and in the general case of the dimension $d\leq 3$ (see \cite{Ni1}), it has been proved that,
for $h>0$, the problem (\ref{Poiss}) - (\ref{Schrod}) admits an
unique solution, $V^{h}$ in our notation, in $H_{0}^{1}(\Omega)$. However, the
uniqueness of $V^{h}$ is by far not obvious if the limit $h\rightarrow0$ is
considered. The aim of our analysis is to understand the asymptotic
behaviour of the system (\ref{Poiss}) - (\ref{Schrod}) as $h\rightarrow0$. This in order to provide a simplified
modelling for the nonlinearities produced by charged particles confined in
quantum layers. Such a program has
been carried out in \cite{FMN} in dimension $d=2$ and $3$. 

In our one dimensional framework, the functional spaces for the solution verify better Sobolev injections than in the $2D$ and $3D$ case, which gives, up to extraction, strong convergence results for the potential (see \cite{BNP1} and \cite{BNP2}). We have then enough regularity to obtain the convergence of the negative spectrum. This allows to determine the limit of the Shr\"odinger-Poisson system (\ref{Poiss}) - (\ref{Schrod}), up to extraction, and the convergence of all the sequence follows from the uniqueness, and the explicit computation, of the solution of the limit problem.

Our main results below, whose proofs are given in Section \ref{sec_limSPd=1}, gather the
asymptotic informations for the $1D$ case.

\begin{theorem}\label{asSP}
The nonlinearity $(V^{h})_{h\in(0,h_0]}$ is bounded in $W^{1,\infty}(0,L)$ and tends, strongly in $C^{0,\alpha}(0,L)$, $\forall \alpha\in(0,1)$, to the potential $V_{0}$ given by
\begin{equation}\label{V0expl}
V_{0}(x)=\left\{
\begin{array}
[c]{l}%
\displaystyle\left(\sum_{i\geq1}f(e_{i}+\theta)\right)(1-\frac{x_{0}}{L})x,\quad0<x\leq x_{0}\\[1.65mm]
\displaystyle\left(\sum_{i\geq1}f(e_{i}+\theta)\right)\frac{x_{0}}{L}(L-x),\quad x_{0}<x<L
\end{array}
\right.
\end{equation}
where $\{e_i\}_{1\leq i\leq N}$ is the discrete spectrum of $H_0$, $e_{i\geq N+1}=0$ and $\theta\in(0,\varepsilon_S-e_1)$ is the unique non negative solution of the nonlinear equation $\displaystyle\theta
=x_{0}(1-\frac{x_{0}}{L})\sum_{i\geq 1}f(e_{i}+\theta)$.\\
The density $(n[V^{h}])_{h\in(0,h_0]}$ tends to the mesure 
\begin{equation}\label{muexpl}
\displaystyle\mu = \sum_{i\geq 1}f(e_{i}+\theta)\delta_{x_0}.
\end{equation}
for the weak* topology on the space $\mathcal{M}_b(0,L)$ of bounded mesures on $(0,L)$.
\end{theorem}

Hence for the $1D$ problem, the nonlinear effect produced at the quantum scale
remains visible at the macroscopic scale in the limit $h\rightarrow0$.

\section{Some asymptotics for the linear operator}\label{sec_limlin}

Consider the operator $H_0$ defined in (\ref{H0}). The multiplication operator $U$ is a relatively compact perturbation of $-\frac{d^2}{dx^2}$. Then, the Kato-Reillich and Weyl Theorems imply that $H_0$ is self-adjoint and $\sigma_{ess}(H_0)=[0,+\infty)$.\\
The reader may refer to Proposition 7.4 in \cite{Si} to see that $\sigma_{p}(H_{0})\neq\varnothing$ is always true in the one dimensional case when $U\leq0$ and not identically
zero, and the discret spectrum of $H_0$ is a countable set of negative eigenvalues with multiplicity one and zero as possible accumulation point. In addition, the Proposition 7.5 in \cite{Si} gives the following estimation on the number $N$ of negative eigenvalues of $H_0$:
\[
\displaystyle N \leq 1 + \int_{\R}|x||U(x)|dx,
\]
bound which is finite in our case. It follows that
\begin{equation}\label{sp_H0}
\displaystyle\sigma(H_0)=\{e_1,...,e_N\}\cup[0,+\infty)
\end{equation}
where the $e_i$ are the points of $\sigma_d(H_0)$. We will work with the convention 
\begin{equation}\label{conv_eH0}
\displaystyle e_i=0, \quad \textrm{ for } i\geq N+1
\end{equation}
The present section is devoted to the description, when $h\rightarrow 0$, of the spectrum of the operator
\begin{equation}
\displaystyle H_{0}^{h}=-h^{2}\frac{d^2}{dx^2}+U^{h}\,,\quad D(H_{0}^{h})=H^{2}\cap H_{0}^{1}(\Omega)
\label{H0_h}%
\end{equation}
using a comparison with the spectrum of $H_0$. The operator $H_0^h$ is the linear part of the Hamiltonian $H^h$ involved in the Schr\"odinger-Poisson system (\ref{Poiss}) - (\ref{Schrod}) and, therefore, it will play an important role in the analysis: the spectral asymptotics of the linear operator will give information about the nonlinear one.\\
We will denote by $(-h^2\frac{d^2}{dx^2})_{\Omega}$ the realisation of $-h^2\frac{d^2}{dx^2}$ on $\Omega$ with domain $H^1_0\cap H^2(\Omega)$. The spectrum of the operator $(-h^2\frac{d^2}{dx^2})_{\Omega}$ is the sequence of the $\lambda_i^h = \frac{h^2i^2\pi^2}{L^2}$ and the corresponding normalized eigenvectors are the $\varphi_i(x)=\sqrt{\frac{2}{L}}\sin\frac{i\pi x}{L}$ for $i\geq 1$.\\
By setting
\[
\displaystyle\Lambda := \left\Vert U\right\Vert _{L^{\infty}} \, > \, 0.
\]
we have:
\[
\displaystyle\sigma((-h^2\frac{d^2}{dx^2}-\Lambda)_{\Omega}) = \{\lambda_i^h-\Lambda, \, i\geq 1\}
\]
Let $\varepsilon$ be a constant such that $\varepsilon>-\Lambda$, we have:
\[
\displaystyle\#\left(\sigma((-h^2\frac{d^2}{dx^2}-\Lambda)_{\Omega})\cap(-\infty,\varepsilon]\right)=\#\{i\geq 1;\,\frac{h^2i^2\pi^2}{L^2}-\Lambda \leq \varepsilon\}\leq \sqrt{\varepsilon+\Lambda}\frac{L}{\pi h}
\]
Coming back to the operator $H_0^h$, we note that:
\[
\displaystyle H_0^h \geq (-h^2\frac{d^2}{dx^2}-\Lambda)_{\Omega}
\]
and therefore $\forall \varepsilon > -\Lambda$, the integer
\begin{equation}\label{eqNh}
\displaystyle N^h := \#\left(\sigma(H_0^h)\cap(-\infty,\varepsilon]\right)
\end{equation}
verifies
\begin{equation}\label{eq0nb}
\displaystyle N^h \leq \#\left(\sigma((-h^2\frac{d^2}{dx^2}-\Lambda)_{\Omega})\cap(-\infty,\varepsilon]\right)=\mathcal{O}(h^{-1}).
\end{equation}
The previous asymptotic order is an important point in the proof of the following Lemma, which is the main result of this section.
\begin{lemma}\label{cv_H0bnb}
Consider $\varepsilon \in (e_N,0)$ and $N^h$ given by (\ref{eqNh}). If we define $e_i^h$ to be the eigenvalues of $H_0^h$, then, for $h_0$ small enough, we have:
\begin{itemize}
\item[$\bullet$]$N^h=N$, $\forall h\in(0,h_0]$
\item[$\bullet$]For $i=1,...,N$:
\[
\lim_{h\rightarrow 0} e_i^h = e_i
\]
\end{itemize}
\end{lemma}
\begin{remark}
\textrm{A consequence of this Lemma is: 
\[
\displaystyle\liminf_{h \rightarrow 0} e_i^h \geq 0
\]
for $i\geq N+1$. In other words, the eigenvalues of $H_0^h$, with number $i\geq N+1$, are asymptotically embedded in the continuous spectrum of $H_0$.}
\end{remark}

The proof of this lemma is based on the exponential decay of eigenfunctions given in Appendix \ref{sec_expdec}. The estimates are written with weight functions involving Agmon distances which proprieties are recalled in Appendix \ref{sec_Agmdist}.

\begin{proof}
The states corresponding to energies below $\varepsilon$ are exponentialy decaying outside the support $\omega^h$ of $U^h$, therefore the behaviour of the eigenvalues of $H_0^h$ is well discribed by the comparison with the operator considered on the whole space, $\tilde{H}_0^h$, defined as follows:
\begin{equation}
\displaystyle \tilde{H}_0^h=-h^2\frac{d^2}{dx^2}+U^h;\quad D(\tilde{H}_0^h)=H^{2}(\mathbb{R})
\end{equation}
The operator $\tilde{H}_0^h$ is unitary equivalent to the operator $H_0$ through the unitary map on $L^2(\mathbb{R})$:
\[
\displaystyle \tilde{\phi}(x)=h^{\frac{1}{2}}\phi(hx+x_0)
\]
and \eqref{sp_H0} implies that:
\[
\displaystyle \sigma(\tilde{H}_0^h)=\{e_1,...,e_N\}\cup[0,+\infty)
\]
We will denote by $\phi_i^h$ and $\tilde{\phi}_i^h$ the normalized eigenvectors of $H_0^h$ and $\tilde{H}_0^h$ respectively.

\noindent$\bullet~$ Introduce the family: $u_i^h\in D(H_0^h)$ defined by
\begin{equation}\label{eq.defui}
\displaystyle u_i^h=\chi\tilde{\phi}_i^h, \quad \textrm{ for } i=1,...,N
\end{equation}
 where $\chi\in C_{0}^{\infty}(\mathbb{R})$ is such that: $\chi=1$ on
$B(x_0,\frac{R}{2})$ and $\chi=0$ on $\mathbb{R}\backslash B(x_0,R)$ for a raduis $R>0$ verifying $B(x_0,R)\subset\Omega$.\\
We have the following classical inequality for self-adjoint operators, which can be found in \cite{He}: 
\[
\displaystyle ||u_{i}^{h}||_{L^{2}(\Omega)}d(e_{i},\sigma(H_0^h))\leq||(H_0^h-e_{i})u_{i}^{h}||_{L^{2}(\Omega)}\,. 
\]
We refer to \cite{DiSj} to see that estimation (\ref{Agmon estimates 1}) can be extended to the whole space where the weight function is given by the Agmon distance on $\R$ related to the potential $(U^h-\varepsilon)$. This, combined with inequality (\ref{compDist}) for $\Omega=\R$, gives
\begin{equation}\label{dectoutesp}
\displaystyle h^2\int_{\mathbb{R}}\,\left\vert e^{c_0\frac{|x-x_0|}{h}}\nabla\tilde{\phi}_{i}^h\right\vert ^{2}\,dx + \int_{\mathbb{R}}\,\left\vert e^{c_0\frac{|x-x_0|}{h}}\tilde{\phi}_{i}^h\right\vert ^{2}\,dx\leq C 
\end{equation}
and it follows that $||u_{i}^{h}||_{L^{2}(\Omega)}$ is bounded from below. Indeed, we have:
\begin{align*}
\displaystyle ||u_{i}^{h}||_{L^{2}(\Omega)}^2 &= \int_{B(x_0,R/2)}|\tilde{\phi}_i^h|^2dx + \int_{\R\setminus B(x_0,R/2)}|\chi|^2|\tilde{\phi}_i^h|^2dx\\[1.5mm]
\displaystyle &= \int_{\R}|\tilde{\phi}_i^h|^2dx + \int_{\R\setminus B(x_0,R/2)}(|\chi|^2-1)|\tilde{\phi}_i^h|^2dx = 1 + \int_{\R\setminus B(x_0,R/2)}(|\chi|^2-1)|\tilde{\phi}_i^h|^2dx
\end{align*}
And, the decay estimate (\ref{dectoutesp}) implies that 
\begin{align*}
\displaystyle \left|\int_{\R\setminus B(x_0,R/2)}(|\chi|^2-1)|\tilde{\phi}_i^h|^2dx\right|&\leq (||\chi||_{L^{\infty}}^2+1)\int_{\R\setminus B(x_0,R/2)}e^{-2c_0|x-x_0|/h}|e^{c_0|x-x_0|/h}\tilde{\phi}_i^h|^2dx\\[1.5mm]
\displaystyle &\leq C\;\int_{\R}|e^{c_0|x-x_0|/h}\tilde{\phi}_i^h|^2dx\;e^{-c_0R/h} \leq C'e^{-c_0R/h}
\end{align*}
which leads to 
\[
\displaystyle ||u_{i}^{h}||_{L^{2}(\Omega)}\geq\frac{1}{2}, \quad \forall h\in(0,h_0]
\] 
for $h_0$ small enough. Then, a direct computation gives:
\[
\displaystyle (H_0^h-e_{i})u_{i}^{h}=-h^2\chi''\tilde{\phi}_{i}^{h}-2h^2\chi'(\tilde{\phi}_{i}^{h})'.
\]
The r.h.s. in the previous equality being supported in a region where $\tilde{\phi}_i^h$ is exponentially decaying, the estimate (\ref{dectoutesp}) gives:
\begin{equation}\label{eq.reH0}
\displaystyle ||(H_0^h-e_{i})u_{i}^{h}||_{L^2(\Omega)}\leq Ce^{-\frac{\gamma}{h}}
\end{equation}
where $C, \gamma > 0$ doesn't depend on $i$ and $h$.\\
We deduce that for $i=1,...,N$
\[
\displaystyle d(e_{i},\sigma(H_0^h)) \underset{h\rightarrow 0}{\longrightarrow}0
\]
It implies $N^h\geq N$ for all $h\in(0,h_0]$ and $h_0$ small. Then, it's enough to show that $N^h\leq N$ to conclude.

\noindent$\bullet~$ For $i=1,...,N^h$, if we set $v_i^{h}=\chi\phi_i^h\in D(\tilde{H}_0^h)$ using (\ref{Agmon estimates 2}), we can reproduce the calculation of the first point to obtain
\begin{equation}\label{eq_rih}
\displaystyle ||(\tilde{H}_0^h-e_i^h)v_i^{h}||_{L^2(\R)}\leq Ce^{-\frac{\gamma}{h}}
\end{equation}
Now, let $I=[-||U||_{L^{\infty}},\varepsilon]$ and $a>0$ be small enough to have
\[
\displaystyle \sigma(\tilde{H}_0^h)\cap\left(  (I+B(0,2a))\backslash I\right)  =\varnothing\,.
\]
We consider the vector space $E$ spaned by $v_{1}^{h},...,v_{N^h}^{h}$ and the spectral subspace $F$ corresponding to $\sigma(\tilde{H}_0^h)\cap I$. Estimation (\ref{Agmon estimates 2}) implies that the matrix $M=((v_{i}^{h},v_{j}^{h})_{L^{2}(\mathbb{R})})_{1\leq i,j\leq N^h}$ verifies
\begin{equation}\label{apId}
\displaystyle M=I+\mathcal{O}(e^{-\frac{\gamma}{h}})
\end{equation}
when $h\rightarrow0$. As it will be clarified in Remark \ref{linInd}, equation (\ref{apId}) implies that the $v_{i}^{h}$ are linearly independant. If we consider in addition equation (\ref{eq_rih}), the conditions for the application of Proposition 2.5 in \cite{HeSj} are gathered and the distance $d(E,F)$ (definition given in
\cite{HeSj}) is estimated by
\[
\displaystyle d(E,F) \leq \left(  \frac{N^h}{\lambda_{min}}\right)^{\frac{1}{2}}%
\frac{Ce^{-\frac{\gamma}{h}}}{a}
\]
where $\lambda_{min}$ is the smallest eigenvalue of $M$. We deduce from (\ref{apId}) that $\lambda_{min}=1+o(1)$, and using (\ref{eq0nb}), we get:
\[
\displaystyle d(E,F)\leq C (N^h)^{\frac{1}{2}}e^{-\frac{\gamma}{h}} \leq Ch^{-\frac{1}{2}}e^{-\frac{\gamma}{h}} \leq \frac{1}{2}
\]
for all $h\in(0,h_0]$ and $h_0$ small enough. This last condition allows us to
state that the projection $\left.  \Pi_{F}\right\vert _{E}:E\rightarrow F$ is
injective (e.g. in Lemma 1.3 in \cite{HeSj}), from which the condition $N^h\leq N$ follows.
\end{proof}

\begin{remark}\label{linInd}
\textrm{The matrix $M:=((v_i^h,v_j^h))$ verifies $M=I+\mathcal{O}(e^{-\frac{\gamma}{h}})$ for a constant $\gamma>0$ where the asymptotics $\mathcal{O}$ is considered with respect to the norm 
\[
\displaystyle ||A||_{\Delta}=\max_{i,j}|a_{ij}|
\]
This implies $M\in GL_{N^h}(\mathbb{C})$ and the linear independence of the $v_i^h$. Indeed, remarking that (see in \cite{GoVa}) 
\[
\displaystyle \forall A\in\mathbb{C}^{N^h\times N^h}, \quad ||A||_2\leq N^h||A||_{\Delta}
\]
where
\[
\displaystyle ||A||_2 = \sup_{x\neq 0}\frac{||Ax||_2}{||x||_2}, \quad ||x||_2=(\sum_ix_i^2)^{\frac{1}{2}}
\]
we obtain $||M-I||_{2}\leq Ch^{-1}e^{-\frac{\gamma}{h}}<1$, $\forall h\in(0,h_0]$ for $h_0$ small. The inequality $||AB||_{2}\leq ||A||_{2}||B||_{2}$ implies that the sequence $S:=\sum_{n\geq 0}(-1)^n(M-I)^n$ converges and verifies:
\[
\displaystyle SM=MS=I
\]
and the matrix $M$ is invertible. Now, suppose there exists $\lambda=(\lambda_i)_i$ such that $\sum_i\lambda_iv_i^h=0$, then $\lambda^TM=0$ and $\lambda=0$.\\
The asymptotics $\lambda_{min}=1+o(1)$, where $\lambda_{min}$ is the smallest eigenvalue of $M$, is given by the perturbation estimate for eigenvalues
\[
\displaystyle |\lambda_n(M)-1|\leq ||M-I||_2, \quad  \sigma(M)=\{\lambda_n(M),~ 1\leq n \leq N^h\}
\]
which can be found in \cite{GoVa}. As $||M-I||_2\leq N^h||M-I||_{\Delta}\leq C e^{-\frac{\gamma}{2h}}$, $\forall h\in(0,h_0]$, we get:
\[
\displaystyle \lambda_{min}=1+\mathcal{O}(e^{-\frac{\gamma}{2h}})
\]
}
\end{remark}

\section{Preliminary results for the Schr\"odinger-Poisson system}\label{sec_premres}

We give some apriori estimates for our Schr\"odinger-Poisson system (\ref{Poiss}) - (\ref{Schrod}). The results presented below are also true in dimension $d\leq 3$ as it appears in \cite{FMN}.\\ 
 From the lower bound
\[
\displaystyle -\frac{d^2}{dx^2} V^{h}\geq0\quad\text{in}\ \Omega
\]
with homogeneous Dirichlet boundary conditions, the maximum principle implies
\begin{equation}
\displaystyle V^{h}\geq0\quad\text{in}\ \Omega  \label{PdM 1}%
\end{equation}
Thus, $V^{h}$ defines a positive perturbation of the Hamiltonian $H_{0}^{h}$ given by (\ref{H0_h}). The spectra of $H_{0}^{h}$ is bounded
from below by the norm $\left\Vert U\right\Vert _{L^{\infty}(\mathbb{R})}$
and we can state%
\begin{equation}
\displaystyle \inf\left\{  \varepsilon_{i}^{h}\right\}  _{i\geq 1}\geq-\left\Vert
U\right\Vert _{L^{\infty}}\,. \label{condition 3}%
\end{equation}
Due to the definition of the source term (\ref{source}), $V^{h}$ is generated by those energy levels $\varepsilon_{i}^{h}$
placed below the cut off $\varepsilon_{S}$ of the characteristic function $f$.
In order to study the semiclassical behaviour of our system, we are interested
into the spectral properties of the Hamiltonian
\begin{equation}
\displaystyle H^{h}=-h^{2}\frac{d^2}{dx^2}+U^{h}+V^{h}\,,\quad D(H^{h})=H^{2}\cap H_{0}^{1}(\Omega)
\label{H_h}%
\end{equation}
in the spectral interval $\left[  -\left\Vert U\right\Vert _{L^{\infty}%
},\,\varepsilon_{S}\right)  $, as $h\rightarrow0$. In particular, a uniform
bound for the number of eigenvalues $\varepsilon_{i}^{h}\in\left[  -\left\Vert
U\right\Vert _{L^{\infty}},\,\varepsilon_{S}\right)  $ as $h\rightarrow0$ is
required. As noticed above, the operator $H^{h}$ is obtained as a
positive perturbation of $H_{0}^{h}$ through the Poisson potential $V^{h}$. Then, the minimax principle implies that%
\begin{equation}
\displaystyle e_{i}^{h}\leq\varepsilon_{i}^{h},\quad\forall i\geq 1
\label{condition 3.1}%
\end{equation}
and the Lemma \ref{cv_H0bnb} leads to the following result.

\begin{lemma}
\label{Lemma 1}Let $N$ be the integer given by \eqref{sp_H0}, then $\forall
h\in\left(  0,h_{0}\right]  $%
\begin{equation}
\displaystyle \#\left(  \sigma(H_{0}^{h})\cap\left[  -\left\Vert U\right\Vert _{L^{\infty}},
\varepsilon_{S}\right)  \right)  \leq N \label{N_0}%
\end{equation}%
\begin{equation}
\displaystyle \#\left(  \sigma(H^{h})\cap\left[  -\left\Vert U\right\Vert _{L^{\infty}
},\varepsilon_{S}\right)  \right)  \leq N \label{N}%
\end{equation}
where $\sigma(H)$ denotes the spectrum of $H$.
\end{lemma}

We recall the variational formulation, given in \cite{Ni1} in dimension $d\leq3$, of the Schr\"odinger-Poisson problem (\ref{Poiss}) - (\ref{Schrod}). Rephrasing the results of
this work for our system, we can state that the solution to the equation
(\ref{Poiss}) - (\ref{Schrod}) is equivalent to the minimization
problem%
\begin{equation}
\displaystyle \inf_{V\in H_{0}^{1}(\Omega)}J(V);\quad J(V)=\frac{1}{2}\int_{\Omega}|V'(x)|^{2}dx+Tr\left[  F\left(  H^{h}(V)\right)  \right]  \,, \label{Francis 1}%
\end{equation}
where $F$ is the positive function%
\[
\displaystyle F(x)=\int_{x}^{+\infty}f(s)ds\,, \label{Francis 2}%
\]
while the Hamiltonian $H^{h}(V)$ is given by%
\[
\displaystyle H^{h}(V)=-h^{2}\frac{d^2}{dx^2}+U^{h}+V\,,\quad D(H^{h}(V))=H^{2}\cap H_{0}^{1}%
(\Omega)\,. \label{Francis 3}%
\]
Moreover, the function $J(V)$ is Fr\'{e}chet-$C^{\infty}$ w.r.t. $V$, strictly
convex and coercive on $H_{0}^{1}(\Omega)$ and
(\ref{Francis 1}) admits a unique solution in this space. The following
Proposition is a direct consequence of this result.

\begin{proposition}
\label{Proposition 1}The solution to the Schr\"{o}dinger-Poisson problem
(\ref{Poiss}) - (\ref{Schrod}) is bounded in $H_{0}^{1}(\Omega)$
uniformly with respect to $h$.
\end{proposition}

\begin{proof}
 From the variational formulation recalled above, the solution $V^{h}$ is the
minimum of the convex map $J(V)$, therefore we have%
\[
\displaystyle \frac{1}{2}\int_{\Omega}|(V^{h})'(x)|^{2}dx+Tr\left[  F\left(  H^{h}%
(V^{h})\right)  \right]  \leq J(0)=Tr\left[  F\left(  H^{h}(0)\right)
\right]  \,,
\]
where $H^{h}(V^{h})$ simply coincides with the Hamiltonian $H^{h}$, while
$H^{h}(0)$ can be identified with $H_{0}^{h}$ defined in (\ref{H0_h}). The
relation
\[
\displaystyle Tr\left[  F\left(  H^{h}(V^{h})\right)  \right]  =\sum_{i=1}^NF(\varepsilon_{i}^{h})\geq0\,,
\]
with $N$ given in Lemma \ref{Lemma 1}, implies%
\begin{equation}
\displaystyle \left\Vert V^{h}\right\Vert _{H_{0}^{1}(\Omega)}^{2}\leq2Tr\left[  F\left(
H_{0}^{h}\right)  \right]  \,. \label{Francis 3.1}%
\end{equation}
From Lemma \ref{cv_H0bnb}, the explicit expression of the r.h.s. here is%
\[
\displaystyle Tr\left[  F\left(  H_{0}^{h}\right)  \right]  =\sum_{i=1}^NF(e_{i}%
^{h})\,.
\]
The result easily follows by combining (\ref{Francis 3.1}) with the inequality%
\[
\displaystyle \sum_{i=1}^NF(e_{i}^{h})\leq N\sup_{x\in\left[  -\left\Vert
U\right\Vert _{L^{\infty}},\varepsilon_{S}\right)  }F<\infty\,.
\]

\end{proof}

Next we use the assumption (\ref{condition 4}) and Lemma \ref{cv_H0bnb} to get uniform upper bound for the first spectral point of $H^{h}$ as $h\rightarrow0$.

\begin{lemma}
\label{Lemma 1.1}For $h_{0}$ small enough, the condition%
\begin{equation}
\displaystyle \varepsilon_{1}^{h}<\varepsilon_{S} \label{L1.1 0}%
\end{equation}
holds for all $h\in\left(  0,h_{0}\right]  $.
\end{lemma}

\begin{proof}
We use a reductio ad absurdum argument. Let $\bar{h}\in(0,h_{0}]$ be such that
$\varepsilon_{1}^{\bar{h}}\geq\varepsilon_{S}$. It follows from
(\ref{condition 2}) and from the definition (\ref{source}) that the
corresponding charge density, $n\left[  V^{\bar{h}}\right]  $, and, then, the
Poisson potential $V^{\bar{h}}$ are null in $\Omega$. In these conditions the
Hamiltonians $H^{\bar{h}}$ and $H_{0}^{\bar{h}}$ coincide and we have%
\begin{equation}
\displaystyle \varepsilon_{1}^{\bar{h}}=e_{1}^{\bar{h}}\geq\varepsilon_{S}\,. \label{L1.1 1}%
\end{equation}
On the other hand, as it follows from Lemma \ref{cv_H0bnb}, we have $e_{1}^{h}\longrightarrow\inf
\sigma\left(  H_{0}\right)$ when $h\rightarrow0$. Then from the assumption
(\ref{condition 4}), the condition%
\begin{equation}
\displaystyle e_{1}^{\bar{h}}<\varepsilon_{S} \label{L1.1 2}%
\end{equation}
definitely holds for $\bar{h}\rightarrow0$, which is in contradiction with
(\ref{L1.1 1}).
\end{proof}

\section{The semi-classical limit}\label{sec_limSPd=1}

In this section, the asymptotic behaviour of the Schr\"odinger-Poisson system (\ref{Poiss}) - (\ref{Schrod}) is described. Contrary to the other sections, the results presented here can not be extended to higher dimensions. Therefore, we will use the notation $(0,L)$ introduced in Section \ref{mod_limSP} to design the open domain $\Omega$.\\ 
Due to (\ref{condition 2}), only a finite number of eigenvectors contributes to the density and therefore $\displaystyle n[V^h]=\sum_{i\geq 1}f(\varepsilon_i^h)|\Psi_i^h|^2$ belongs to $L^1(0,L)$ as a finite sum of $L^1$ functions. Then, taking into account the embedding $H^1(0,L)\subset C(0,L)$ in dimension $1$, the accurate function space for the potential is:
\[
\displaystyle BV_0^2(0,L):= \{ V\in C(0,L);~ V''\in \mathcal{M}_b(0,L), \, V(0)=V(L)=0\} 
\]
The space $BV_0^2(0,L)$ with the norm:
\[
\displaystyle ||V||_{BV}:=||V||_{C(0,L)}+||V''||_{m}
\] 
is a Banach space. Here $||\mu||_{m}=|\mu|(0,L)$ is the strong norm of the bounded mesure $\mu$ which is equal to the $L^1$ norm when $\mu\in L^1(0,L)$.\\ 
We have the following continuous injections which will be useful in this work (e.g. in \cite{BNP1}, \cite{BNP2}): $\forall \alpha \in (0,1)$
\begin{equation}\label{injBV}
\displaystyle BV_0^2(0,L)\hookrightarrow W^{1,\infty}(0,L), \quad BV_0^2(0,L)\subset\subset C^{0,\alpha}(0,L) 
\end{equation}
where the second injection is compact. The continuous injection $BV_0^2(0,L)\hookrightarrow W^{1,\infty}(0,L)$ doesn't present real difficulties, the distribution function of a bounded mesure being regular. The second injection is a consequence of the compact embedding $W^{1,\infty}(0,L)\subset\subset C^{0,\alpha}(0,L)$.\\  
We can already give some apriori estimates for the potential and the density.
\begin{proposition}\label{propEst}
The density $(n[V^h])_{h\in(0,h_0]}$ is bounded in $\mathcal{M}_b(0,L)$. The potential $(V^h)_{h\in(0,h_0]}$ is bounded in $W^{1,\infty}(0,L)$ and relatively compact in $C^{0,\alpha}(0,L)$ for any $\alpha \in (0,1)$. 
\end{proposition}
\begin{proof}
We start looking for an estimate on the density. Using the normalization condition (\ref{condition eigenvectors 1}), we have:
\begin{eqnarray*}
\displaystyle ||n[V^h]||_m &=& \int_0^L n[V^h] dx = \sum_{i\geq 1}f(\varepsilon_i^h)||\Psi_i^h||_{L^2(0,L)}^2\\[1.5mm]
\displaystyle &\leq& N^h \sup_{x\in[-||U||_{\infty},\varepsilon_S)}f(x)
\end{eqnarray*}
where $N^h$ denotes the number of eigenvalues of $H^h$ below
$\varepsilon_S$. According to Lemma \ref{Lemma 1}, we have $N^h=\mathcal{O}(1)$ when $h$ tends to $0$ and the density is bounded in $\mathcal{M}_b(0,L)$.\\
For what concerns the potential $V^h$, the $H^1_0$ bound provided by Proposition \ref{Proposition 1} and the continuous embedding $H^1_0(0,L)\hookrightarrow C(0,L)$ implies
\[
\displaystyle ||V^h||_{C(0,L)} \leq C, \quad \forall h\in(0,h_0]
\] 
Then, it follows from the Poisson equation $-(V^h)''= n[V^h]$ and the bound on the density that
\[
\displaystyle ||(V^h)''||_{m}= ||n[V^h]||_m \leq C, \quad \forall h\in(0,h_0]
\]
We obtain that $(V^h)_{h\in(0,h_0]}$ is bounded in $BV_0^2(0,L)$ and (\ref{injBV}) allows to conclude.
\end{proof}

The density $(n[V^h])_{h\in(0,h_0]}$ being bounded in $\mathcal{M}_b(0,L)$, it is relatively compact in the weak* topology. Therefore, we deduce from Proposition \ref{propEst} that for any $0<\alpha<1$, it is possible to extract from any infinite set $S\subset(0,h_0]$, which has $0$ as an accumulation point, a countable subset $D$ such that $0\in\overline{D}$ and:
\begin{equation}\label{cvExt1}
 \displaystyle \lim_{\substack{h\rightarrow0\\h\in D}} (n[V^h]-\mu,\varphi) = 0, \quad \forall \varphi\in C([0,L])
\end{equation}
\begin{equation}\label{cvExt2}
\displaystyle \lim_{\substack{h\rightarrow0\\h\in D}}||V^h-V_0||_{0,\alpha} = 0,
\end{equation}
for some $\mu\in\mathcal{M}_b(0,L)$ and $V_0\in C^{0,\alpha}(0,L)$. Here, $||.||_{0,\alpha}$ is the usual norm on the H\"older space $C^{0,\alpha}(0,L)$. Then we have the following result on the convergence of the spectrum of $H^h$:
\begin{lemma}\label{cvSP_NL}
Consider $0<\alpha<1$, $D \subset (0,h_0]$ s.t. $0\in\overline{D}$, and $V_0\in C^{0,\alpha}(0,L)$
 verifying (\ref{cvExt2}). If $\theta=V_0(x_0)$ and $N_1$ is the greatest integer such that $e_{N_1}+\theta < 0$, then:
\begin{itemize}
\item[$\bullet$] $\theta\geq 0$ and $N_1\geq 1$
\item[$\bullet$] For $i=1,...,N_1$ 
\[
\displaystyle \lim_{\substack{h\rightarrow0\\h\in D}}\varepsilon_i^h=e_i+\theta
\]
 \item[$\bullet$] For $i\geq N_1+1$ 
\[
\displaystyle \liminf_{\substack{h\rightarrow0\\h\in D}}\varepsilon_i^h\geq 0
\]
\end{itemize}
\end{lemma}
The proof of this lemma is similar to the proof of Lemma \ref{cv_H0bnb} where the nonlinearity have to be considered as it is done in \cite{Pa} and \cite{FMN} in the two dimensional case.

\begin{proof}
The asymptotics of the $\varepsilon_i^h$ is given by the comparison with the Hamiltonian:
\[
\displaystyle \tilde{H}_\theta^h = -h^2\frac{d^2}{dx^2}+U^h+\theta, \quad D(\tilde{H}_\theta^h)=H^2(\R)
\]
The operator $\tilde{H}_\theta^h$ is unitary equivalent to $H_\theta$ defined by:
\[
\displaystyle H_\theta = -\frac{d^2}{dx^2}+U+\theta, \quad D(H_\theta)=H^2(\R)
\]
and therefore:
\[
\displaystyle \sigma(\tilde{H}_\theta^h)=\{e_1^{\theta},...,e_N^{\theta}\}\cup[\theta,+\infty[
\]
where $e_i^{\theta} := e_i+\theta$. We note that the normalized eigenvectors of $\tilde{H}_\theta^h$ are the $\tilde{\phi}_i^h$ introduced in the proof of Lemma \ref{cv_H0bnb}.\\
The $W^{1,\infty}$ bound obtained in Proposition \ref{propEst} implies that there exists a constant $C>0$ such that $\forall h\in (0,h_0]$:
\begin{equation}\label{bornW1inf}
\displaystyle |V^h(x)-V^h(y)|\leq C|x-y|, \quad \forall x,y \in \Omega
\end{equation}
and the condition (\ref{cvExt2}) gives:
\begin{equation}\label{cvVTet}
\displaystyle \lim_{\substack{h\rightarrow0\\h\in D}}|V^h(x_0)-\theta| = 0
\end{equation}
The previous equations are important in our nonlinear framework. Indeed, they asymptotically enable to consider $V^h$ as a constant on a domain outside which the modes of interest are exponentially decaying.\\ 
Note that $\theta\geq 0$ is a direct consequence of (\ref{PdM 1}) and (\ref{cvVTet}).\\

\noindent$\bullet~$ Set $\varepsilon \in (-||U||_{L^{\infty}},0)$ and define:
\begin{equation}\label{defNh_NL}
\displaystyle N^h=\#(\sigma(H^h)\cap(-\infty,\varepsilon])
\end{equation}
We introduce the family: $v_i^h=\chi\Psi_i^h\in D(\tilde{H}_\theta^h)$, for $i=1,...,N^h$, where $\chi\in C_{0}^{\infty}(\mathbb{R})$ is s.t. $\chi=1$ on
$B(x_0,\frac{R}{2})$ and $\chi=0$ on $\mathbb{R}\backslash B(x_0,R)$ for a radius $R>0$ verifying $B(x_0,R)\subset\Omega$.\\
As seen in the proof of Lemma \ref{cv_H0bnb}, the exponential decay (\ref{Agmon estimates 2}) for $\Psi_i^h$ implies that it is enough to estimate $(\tilde{H}_\theta^h-\varepsilon_{i}^{h})v_{i}^{h}$ to get 
\begin{equation}\label{distT0}
\displaystyle d(\varepsilon_{i}^{h},\sigma(\tilde{H}_\theta^h))1_{i\leq N^h}\underset{h\rightarrow 0, \, h\in D}{\longrightarrow}0
\end{equation}
Now we have
\[
\displaystyle (\tilde{H}_\theta^h-\varepsilon_{i}^{h})v_{i}^{h}=-h^2\chi''\Psi_i^h-2h^2\chi'(\Psi_i^h)'-\chi(V^h-\theta)\Psi_i^h
\]
The function $\chi$ being regular with derivatives supported where the eigenfunction is exponentially decaying, we have:
\[
\displaystyle ||h^2\chi''\Psi_i^h+2h^2\chi'(\Psi_i^h)'||_{L^2(\R)}\leq Ce^{-\frac{\gamma}{h}}
\]
for $C, \gamma > 0$. Next, we define $R^{h}=h\ln\frac{1}{h}$ and we suppose that $h_{0}$ is small enough so that $B(x_0,R^h)\subset\Omega$. Proposition \ref{propEst} implies that the potential $V^h$ is bounded in $L^{\infty}(\Omega)$, therefore, the application of the estimate (\ref{Agmon estimates 2}) leads to: 
\begin{align*}
\displaystyle \left\Vert \chi(V^h-\theta)\Psi_{i}^{h}\right\Vert
_{L^{2}(\mathbb{R})}^{2}  &  \leq||\chi||_{L^{\infty}(\mathbb{R})}^{2}\left(  \int_{B(x_0,R^{h})}|(V^{h}-\theta)\Psi_{i}^{h}%
|^{2}dx+\int_{\Omega\backslash B(x_0,R^{h})}|(V^{h}-\theta
)\Psi_{i}^{h}|^{2}dx\right) \\[1.5mm]
\displaystyle &  \leq C\left(  \,||V^{h}-\theta||_{L^{\infty}(B(x_0,R^{h}))}%
^{2}\,+\,\int_{\Omega}|e^{\frac{c_{0}|x-x_0|}{h}}\Psi_{i}^{h}|^{2}%
dx\,e^{-2c_{0}\frac{R^{h}}{h}}\,\right)\\[1.5mm]
\displaystyle &\leq C||V^{h}-\theta||_{L^{\infty}(B(x_0,R^{h}))}^{2}+\mathcal{O}(e^{-2c_{0}\ln\frac{1}{h}})
\end{align*}
On the other hand, we deduce from equations (\ref{bornW1inf}) and (\ref{cvVTet}) that $\forall x\in B(x_0,R^{h})$
\[
\displaystyle |V^{h}(x)-\theta| \leq |V^{h}(x)-V^{h}(x_0)|+|V^{h}(x_0)-\theta|\leq Ch\ln\frac{1}{h}+o(1)
\]
when $h\rightarrow 0$, $h\in D$. Here, the asymptotics $o(1)$ doesn't depend on $x$ and therefore
\[
\displaystyle \lim_{\substack{h\rightarrow0\\h\in D}}\left\Vert \chi(V%
^{h}-\theta)\Psi_{i}^{h}\right\Vert_{L^{2}(\mathbb{R})}1_{i\leq N^h}=0\,.
\]
We remark then that a function $\alpha(h)$ independant of $i$ can be found such that $\alpha(h)\rightarrow0$ when $h\rightarrow0$ and
\begin{equation}\label{eqRih_NL}
\displaystyle ||(\tilde{H}_\theta^h-\varepsilon_{i}^{h})v_{i}^{h}||_{L^{2}(\mathbb{R})}1_{i\leq N^h}%
\leq\alpha(h),\quad\forall h\in D\,.%
\end{equation} 
and (\ref{distT0}) follows.\\
According to Lemma \ref{Lemma 1.1}, we have $\varepsilon_1^h<\varepsilon_S$, therefore the application of (\ref{distT0}) with
$\varepsilon =\varepsilon_S$ implies $e_1^\theta<0$ and $N_1 \geq 1$. This provides the first point of the Lemma.\\

\noindent$\bullet~$ Consider the family $u_i^h\in D(H^h)$ given by \eqref{eq.defui}. Then, for $i=1,...,N_1$, we have:
\[
\displaystyle (H^h-e_{i}^\theta)u_{i}^{h}=-h^2\chi''\tilde{\phi}_i^h-2h^2\chi'(\tilde{\phi}_i^h)'+\chi(V^h-\theta)\tilde{\phi}_i^h
\]
Using the exponential decay estimate (\ref{dectoutesp}) for $\tilde{\phi}_i^h$, the same computations as in the previous point can be performed to obtain
\begin{equation}\label{distT0_opp}
\displaystyle d(e_i^\theta,\sigma(H^h))\underset{h\rightarrow 0,\, h\in D}{\longrightarrow}0
\end{equation}

In what follows, we set $\varepsilon\in (e_{N_1}^\theta,0)$ and $N^h$ the corresponding integer given by relation (\ref{defNh_NL}).\\
 From Lemma \ref{cv_H0bnb}, we have $N^h=\mathcal{O}(1)$. Then, taking into account the results (\ref{distT0}) and (\ref{distT0_opp}) above, we can already deduce that the following statement holds: for all $\delta >0$ small, $\exists h_0$ s.t. $\forall h\in D\cap(0,h_0]$ we have, even if the numerotation of the eigenvalues has to be changed,
\begin{equation}\label{eqInt_NL}
\displaystyle \varepsilon_{k,1}^h,...,\varepsilon_{k,N_k}^h \in [e_k^\theta-\delta,e_k^\theta+\delta] \quad \textrm{ for } \quad k=1,...,N_1
\end{equation}
where $N_k\geq 1$ is an integer depending on $h$ such that $\sum_{k=1}^{N_1} N_k = N^h$. This implies $N^h\geq N_1$ and if we obtain $N^h \leq N_1$, $\forall h\in D\cap(0,h_0]$, the proof will be completed.

\noindent$\bullet~$ Consider $I=[-||U||_{L^{\infty}},\varepsilon]$ and $a>0$ small enough so that
\[
\displaystyle \sigma(\tilde{H}_\theta^h)\cap\left(  (I+B(0,2a))\backslash I\right)  =\varnothing\,.
\]
Let $E$ be the vector space spanned by $v_{1}^{h},...,v_{N^h}^{h}$ and $F$ the spectral subspace associated to $\sigma(\tilde{H}_\theta^h)\cap I$. From the estimate (\ref{Agmon estimates 2}), the matrix $M=((v_{i}^{h},v_{j}^{h})_{L^{2}(\R)})_{1\leq i,j\leq N^h}$ verifies
\[
\displaystyle M=I+\mathcal{O}(e^{-\frac{\gamma}{h}})
\]
when $h\rightarrow0$ and the $v_{i}^{h}$ are linearly independant. Then, equation (\ref{eqRih_NL}) and Proposition 2.5 in \cite{HeSj} give
\[
\displaystyle d(E,F) \leq \left(  \frac{N^h}{\lambda_{min}}\right)^{\frac{1}{2}}%
\frac{\alpha(h)}{a}, \quad \forall h\in D
\]
where $\lambda_{min}$ is the smallest eigenvalue of $M$. As we already noticed, $N^h=\mathcal{O}(1)$ and therefore:
\[
\displaystyle d(E,F)\leq C \alpha(h) \leq \frac{1}{2}
\]
for all $h\in D\cap(0,h_0]$ with $h_0$ small enough. Then, Lemma 1.3 in \cite{HeSj} gives $N^h\leq N_1$.
\end{proof}

In what follows, we give, for a set $D$ verifying (\ref{cvExt1}) and (\ref{cvExt2}), the limit problem when $h\rightarrow 0$. Then the unicity of the limit allows to determine the asymptotics of the solution of (\ref{Poiss}) - (\ref{Schrod}) when $h\rightarrow 0$.
\begin{proposition}\label{prop_pbLim}
Consider $0<\alpha<1$, $D \subset (0,h_0]$ s.t. $0\in\overline{D}$, $\mu\in\mathcal{M}_b(0,L)$ and $V_0\in C^{0,\alpha}(0,L)$ verifying (\ref{cvExt1}) and (\ref{cvExt2}). Then, we have 
\begin{equation}\label{val_mu}
\displaystyle \mu = \sum_{i\geq 1}f(e_i+V_0(x_0))\delta_{x_0}
\end{equation}
and the potential $V_0$ is solution of the problem
\begin{equation}\label{pblim}
\left\{\begin{array}{ll} \displaystyle -\frac{d^2}{dx^2}V_0=\mu,&(0,L)\\[1.65mm]
\displaystyle V_0(0)=V_0(L)=0&
\end{array}\right.
\end{equation}
\end{proposition}
\begin{proof}
Set $\theta = V_0(x_0)$ and take $\varepsilon\in (e_{N_1}+\theta,0)$ where $N_1$ is the integer defined in Lemma \ref{cvSP_NL}. Then, from Lemma \ref{cvSP_NL}, we have for $h_0$ small enough:
\[
\displaystyle \#(\sigma(H^h)\cap(-\infty,\varepsilon])=N_1, \quad \forall h \in D\cap(0,h_0]
\]
and the density can be written as follows:
\[
\displaystyle \sum_{i=1}^{N_1}f(\varepsilon_i^h)|\Psi_i^h|^2
\]
where $\varepsilon_i^h\leq\varepsilon$. Using the normalization (\ref{condition eigenvectors 1}) and the exponential decay estimates, we get that $|\Psi_i^h|^2dx$ is of mass $1$ and concentrates around $x_0$. It implies that for $i=1,...,N_1$
\begin{equation}\label{cvDir}
\displaystyle |\Psi_i^h|^2dx \underset{h\rightarrow 0,\, h\in D}{\rightharpoonup} \delta_{x_0}
\end{equation}
for the weak* topology on the space of bounded mesures on $(0,L)$. Indeed, for $i=1,...,N_1$ we can apply the estimate (\ref{Agmon estimates 2}) to write
\begin{align*}
\displaystyle \left\vert \int_0^L\left\vert \Psi_{i}^{h}\right\vert ^{2}\varphi
~dx\right\vert  &  \leq\int_{\mathop{\rm supp}\nolimits\varphi}\left\vert
\varphi\right\vert \,e^{-2c_0|x-x_0|/h}\,\left\vert e^{c_0|x-x_0|/h}\Psi_{i}%
^{h}\right\vert ^{2}\,dx\\[1.5mm]
\displaystyle &  \leq\left\Vert \varphi\right\Vert _{L^{\infty}}\;\left\Vert e^{c_0|x-x_0|/h}%
\Psi_{i}^{h}\right\Vert _{L^{2}(0,L)}^{2}\;\sup_{x\in
\mathop{\rm supp}\nolimits\varphi}e^{-2c_0|x-x_0|/h}\leq C\left\Vert \varphi
\right\Vert _{L^{\infty}}\sup_{x\in\mathop{\rm supp}\nolimits\varphi}%
e^{-2c_0|x-x_0|/h}%
\end{align*}
for any $\varphi\in L^{\infty}(0,L)$ and $\forall h \in D\cap(0,h_0]$. It implies that, when $\mathop{\rm supp}\nolimits\varphi$ is a compact set in $(0,L)\backslash
\left\{  x_{0}\right\}  $, there exists $c_{\varphi}>0$ such that
\[
\displaystyle \left\vert \int_0^L\left\vert \Psi_{i}^{h}\right\vert ^{2}\varphi
~dx\right\vert \leq C\,\left\Vert \varphi\right\Vert _{L^{\infty}%
}e^{-c_{\varphi}/h}, \quad\forall h \in D\cap(0,h_0]
\]
By taking the limit as $h\rightarrow0$, we get%
\[
\displaystyle \lim_{\substack{h\rightarrow0\\h\in D}}\left\vert \int_0^L\left\vert \Psi_{i}^{h}\right\vert
^{2}\varphi~dx\right\vert =0
\]
for all continuous function $\varphi\in C([0,L])$ with
$\mathop{\rm supp}\nolimits\varphi\subset(0,L)\setminus\left\{
x_{0}\right\}  $.

We recall that Lemma \ref{cvSP_NL} also gives for $i=1,...,N_1$:
\[
\displaystyle \lim_{\substack{h\rightarrow0\\h\in D}}\varepsilon_i^h=e_i+\theta
\]
Therefore, from the unicity of the limit (\ref{cvExt1}), we get the result (\ref{val_mu}) using the convergence (\ref{cvDir}), the continuity of $f$ and the convention \eqref{conv_eH0}.\\
On the other hand, the convergence (\ref{cvExt1}), (\ref{cvExt2}) is valid in $\mathcal{D}'(0,L)$:
\[ 
\displaystyle \lim_{\substack{h\rightarrow0 \\ h\in D}} (n[V^h]-\mu,\varphi) = 0, \quad 
\lim_{\substack{h\rightarrow0 \\ h\in D}} (V^h-V_0,\varphi) = 0, \quad \forall \varphi\in C^{\infty}_0(0,L)
\]
From the Poisson equation $-\frac{d^2}{dx^2}V^h=n[V^h]$, $\forall h\in (0,h_0]$ and the continuity of the derivative on $\mathcal{D}'(0,L)$, we get
\[
\displaystyle -\frac{d^2}{dx^2}V_0=\mu, \quad \mathcal{D}'(0,L)
\]
As a consequence of (\ref{cvExt2}), the potential $V^h$ tends to $V_0$ strongly in $C^{0,\alpha}(0,L)$ and the boundary conditions appearing in the problem (\ref{pblim}) follow from 
\[
\displaystyle V_0(x)=\lim_{\substack{h\rightarrow0 \\ h\in D}}V^h(x)
\] 
for $x=0$ and $x=L$.
\end{proof}

We conclude this section giving the proof of Theorem \ref{asSP}.

\begin{proof}
[Proof of Theorem \ref{asSP}]
First, we recall that the $W^{1,\infty}(0,L)$ bound for the potential was obtained in Proposition \ref{propEst}.\\
Next, let $\alpha$ be a constant in $(0,1)$ and consider $D \subset (0,h_0]$ s.t. $0\in\overline{D}$, $\mu\in\mathcal{M}_b(0,L)$ and $V_0\in C^{0,\alpha}(0,L)$ verifying (\ref{cvExt1}) and (\ref{cvExt2}). As a consequence of Proposition \ref{prop_pbLim}, the potential $V_0$ verifies the problem (\ref{val_mu})(\ref{pblim}), and it can be computed explicitly as a function of $\theta := V_0(x_0)$ by solving the following transmission problem:
\[
\left\{
\begin{array}{ll}
\displaystyle -(V_0)'' = 0,& (0,x_0)\cup(x_0,L)\\[1.55mm]
\displaystyle V_0(0)=V_0(L)=0
\end{array}
\right. 
\]
with the jump conditions:
\[
\left\{
\begin{array}{l}
\displaystyle V_0(x_0^-)=V_0(x_0^+)\\[1.55mm]
\displaystyle V_0'(x_0^+)-V_0'(x_0^-)=-(\sum_{i\geq 1}f(e_i+\theta))
\end{array}
\right.
\] 
The formula (\ref{V0expl}) for $V_0$ follows. Then, equation (\ref{V0expl}) considered at $x=x_0$ implies that the value $\theta$ of the potential $V_0$ at $x_0$, is solution of:   
\begin{equation}\label{eqTheta}
\displaystyle \theta - x_{0}(1-\frac{x_{0}}{L})\sum_{i\geq 1}f(e_{i}+\theta) = 0
\end{equation}
According to the first point in Lemma \ref{cvSP_NL}, we have $V_0(x_0)\geq 0$ and we are interested in non negative solutions of (\ref{eqTheta}). If we call $G(\theta)$ the l.h.s. of this equation, we remark that $G$ is a continuous, strictly increasing function on $[0,+\infty)$ and such that $\displaystyle \lim_{\theta \rightarrow +\infty }G(\theta)=+\infty$. Therefore, it defines a bijection from $[0,+\infty)$ to $[G(0),+\infty)$. We have $G(0)=- x_{0}(1-\frac{x_{0}}{L})\sum_{i\geq 1}f(e_{i})<0$ and $G(\varepsilon_S-e_1)=\varepsilon_S-e_1>0$ by the assumption (\ref{condition 4}). We deduce that there exists a unique value $\theta\geq 0$ solving (\ref{eqTheta}), moreover, it verifies $\theta\in (0,\varepsilon_S-e_1)$.\\
It follows that the function $V_0$ (resp. $\mu$) verifying (\ref{cvExt2}) (resp. (\ref{cvExt1})) is given in an unique way by (\ref{V0expl}) (resp. (\ref{muexpl})) where $\theta$ is the positive solution of (\ref{eqTheta}). This gives the convergence results anounced in our theorem.\\
Indeed, suppose that the convergence doesn't occure. Then, there exists a function $\varphi_0\in C([0,L])$, a constant $\varepsilon>0$ and a set $S\subset(0,h_0]$ s.t. $0\in\overline{S}$ verifying:
\begin{equation}\label{suppNCV}
\displaystyle ||V^h-V_0||_{0,\alpha}+|(n[V^h]-\mu,\varphi_0)| \geq \varepsilon, \quad \forall h\in S
\end{equation}
By applying Proposition \ref{propEst}, we can extract a set $D\subset S$ s.t. $0\in\overline{D}$ and (\ref{cvExt1}), (\ref{cvExt2}) are verified for some functions $\tilde{\mu}\in\mathcal{M}_b(0,L)$ and $\tilde{V_0}\in C^{0,\alpha}(0,L)$. Then, according to the previous uniqueness result $\tilde{\mu}=\mu$ and $\tilde{V_0}=V_0$ and we get a contradiction comparing (\ref{cvExt1}), (\ref{cvExt2}) with (\ref{suppNCV}). 
\end{proof}
\begin{remark}
\textrm{In the proof of Theorem \ref{asSP} above, we have shown that the solution $\theta$ of (\ref{eqTheta}) is in the interval $(0,\varepsilon_S-e_1)$. The condition $\theta\geq 0$ implies that the sum $\sum_{i\geq 1}f(e_{i}+\theta)$ is finite. Although the bound $\theta<\varepsilon_S-e_1$ has no impact on the proof, it gives that $\sum_{i\geq 1}f(e_{i}+\theta)>0$, and therefore, that the limit potential $V_0$ and density $\mu$ are not trivial. We deduce that Theorem \ref{asSP} provides a non trivial approximation of the solution $V^h$ of the problem (\ref{Poiss}) - (\ref{Schrod}) in the semi-classical limit $h\rightarrow 0$. As noted in the introduction, it is not the case anymore in dimension $d=2$ and $3$ as it appears in \cite{FMN}.}       
\end{remark}

\begin{description}
\item[Acknowledgements] The author gratefully acknowledges the contribution of
Francis Nier, Florian M\'{e}hats, Andrea Mantile and Naoufel Ben Abdallah, whose advice and remarks aided
in the completion of this study. He wishes to thank the support by the project "QUATRAIN" (No. BLAN07-2 219888, founded by the French National Agency for the Research).
\end{description}

\appendix{}

\section{Agmon distance}\label{sec_Agmdist}

We will mainly follow \cite{HiSi} and present an additional result in the framework of our problem. Let $f$ be a real-valued function, continuous on a bounded connected set $\Omega\subset\R$.\\
For given $x$, $y\in\Omega$, the Agmon distance related to $f$ is:
\[
\displaystyle d(x,y) = \int_0^1(f(\gamma(s)))_+^{\frac{1}{2}}|\gamma'(s)|ds
\]
where $\gamma$ is the segment linking $x$ and $y$.\\
We remark that this distance is degenerated: it may happen that $d(x,y)=0$ with $x\neq y$. However the Agmon distance verifies the following properties: $\forall\, x, y, z\in\Omega$
\[
\displaystyle d(x,y)=d(y,x), \quad d(x,z) \leq d(x,y)+d(y,z) 
\]
For $y\in\Omega$ fixed, the fonction $x\mapsto d(x,y)$ is Lipschitzian continuous. Then, it is differentiable almost everywhere by Rademacher's Theorem and everywhere $x\mapsto d(x,y)$ is differentiable, we have:
\[
\displaystyle |\nabla_x d(x,y)|\leq (f(x))_+^{\frac{1}{2}}
\] 
For a set $\omega\subset\Omega$, we define:
\[
\displaystyle d(x,\omega)=\inf_{y\in \omega}d(x,y)
\]
Then, $d(x,\omega)$ has the same regularity as $x\mapsto d(x,y)$ and
\begin{equation}\label{Agmon distance 1}
\displaystyle |\nabla d(x,\omega)|\leq (f(x))_+^{\frac{1}{2}} \quad\quad a.e.~\Omega
\end{equation}
In our case, we will use the Agmon distance associated to the potential $U^h-\varepsilon$ for a constant $\varepsilon < 0$ which corresponds to taking $f=U^h-\varepsilon$.\\
In the exponential decay estimates, it will be useful to replace the Agmon distance with the Euclidian distance. To do this, the main point is to remark that for $f=1$ the Agmon distance, $d(x,y)$, corresponds to the Euclidian one, $|x-y|$.
\begin{lemma}
Let $\omega^h$ be the support of the well $U^h$, then $\exists c_0, c_1 > 0$ such that:
\begin{equation}\label{compDist}
\displaystyle d(x,\omega^h)\geq c_0|x-x_0| - c_1h, \quad \forall x\in \Omega
\end{equation}
\end{lemma}

\begin{proof}
Let $x\in \Omega\setminus\omega^h$, $y\in\omega^h$ and consider the segment $\gamma:[0,1]\rightarrow\Omega$ linking $x$ and $y$. Then, there exists a unique value $s_0\in (0,1]$ s.t. $\gamma(s_0)\in\partial\omega^h$. Defining $y_0=\gamma(s_0)$, we have: 
\begin{eqnarray*}
\displaystyle \int_0^1(U^h(\gamma(s))-\varepsilon)_+^{\frac{1}{2}}|\gamma'(s)|ds &\geq& \int_0^{s_0}(U^h(\gamma(s))-\varepsilon)_+^{\frac{1}{2}}|\gamma'(s)|ds\\[1.5mm]
\displaystyle &=&|\varepsilon|^{\frac{1}{2}}\int_0^{s_0}|\gamma'(s)|ds = |\varepsilon|^{\frac{1}{2}}|y_0-x|
\end{eqnarray*}
Recalling that $\omega^h\subset B(x_0,h)$, it follows:
\[
\displaystyle d(x,y)\geq |\varepsilon|^{\frac{1}{2}}(|x-x_0|-h)
\]
\end{proof}

\section{Exponential decay of eigenfunctions}\label{sec_expdec}

The exponential decay estimates, also called Agmon estimates, form a standard technical tool in evaluating the rate of decay
of eigenfunctions far from the interaction support. In what follows, we apply
this technique to the case of the Schr\"{o}dinger Poisson system with a
squeezing quantum well; in particular we give some useful decay estimates for
those stationary states related to the energies below some negative energy.
\begin{lemma}
[Agmon estimates]\label{Agmon estimates}Let $\varepsilon^{h}$ be a
spectral point of the Hamiltonian
\[
\displaystyle H^{h}=-h^{2}\Delta+U^{h}+V^{h}%
\]
placed below some negative energy: $\varepsilon^{h}\leq\varepsilon$ where
$\varepsilon\in(-||U||_{L^{\infty}},0)$. The related normalized eigenvector $\Psi^{h}$ admits the estimate
\begin{equation}
\displaystyle \left\Vert h\,\nabla\left(  e^{\phi/h}\Psi^{h}\right)  \right\Vert
_{L^{2}(\Omega)}+\left\Vert e^{\phi/h}\Psi^{h}\right\Vert _{L^{2}(\Omega
)}\leq C\,, \label{Agmon estimates 1}%
\end{equation}%
where $C$ is a suitable positive constant, $\phi$ is the weight function
\[
\displaystyle \phi(x)=(1-\delta)\,d(x,\omega^{h}),\quad x\in
\Omega
\]
$\delta$ is a positive parameter smaller than $1$, $d(x,y)$ is the Agmon distance introduced in Appendix \ref{sec_Agmdist}, while $\omega^{h}$ is the support of $U^{h}$.
\end{lemma}

\begin{proof}
We use the relation (see for instance Theorem 1.1 in \cite{HeSj})%
\[
\displaystyle h^2\int_{\Omega}|\nabla(e^{\frac{\phi}{h}}u)|^{2}dx+\int_{\Omega}(V-|\nabla
\phi|^{2})e^{2\frac{\phi}{h}}\,u^{2}dx=\int_{\Omega}e^{2\frac{\phi}{h}}\left(
-h^2\Delta+V\right)  u\cdot udx
\]
Setting $u=\Psi^{h}$, $\phi=(1-\delta)\,d(x,\omega^{h})$ and $V=U^h+V^{h}-\varepsilon^{h}$, we get%
\[
\displaystyle h^2\int_{\Omega}|\nabla(e^{\frac{\phi}{h}}\Psi^{h})|^{2}%
dx+\int_{\Omega}(U^h+V^{h}-\varepsilon^{h}-|\nabla\phi|^{2})e^{2\frac{\phi}{h}}\,(\Psi^{h})^{2}dx
\,= \, 0
\]
Next we follow the same line as in Proposition 3.3.1 of \cite{He} and
introduce the set%
\[
\displaystyle \Omega_\delta^+ =\left\{  x\in\Omega\left\vert \,U^h-\varepsilon
\geq\delta\right.  \right\}  \,,
\]
\[
\displaystyle \Omega_\delta^- =\left\{  x\in\Omega\left\vert \,U^h-\varepsilon<\delta\right.  \right\}  \,,
\]
where $\delta$ is a positive parameter such that $\delta<\min(1,|\varepsilon|)$. As it appears in \eqref{PdM 1}, we have $V^h\geq 0$ and therefore
\begin{align}
\displaystyle &h^2\int_{\Omega}|\nabla(e^{\frac{\phi}{h}}\Psi^{h})|^{2}%
dx+\int_{\Omega_\delta^+}(U^h+V^{h}-\varepsilon^{h}-|\nabla\phi|^{2})e^{2\frac{\phi}{h}}\,(\Psi^{h})^{2}dx\,=\,
 \int_{\Omega_\delta^-}(|\nabla\phi|^{2}-(U^h+V^{h}-\varepsilon^{h}))e^{2\frac{\phi}{h}}\,(\Psi^{h})^{2}dx\nonumber\\[1.5mm]
\displaystyle &\hspace{2cm} \leq\,
 \int_{\Omega_\delta^-}(|\nabla\phi|^{2}-(U^h-\varepsilon))e^{2\frac{\phi}{h}}\,(\Psi^{h})^{2}dx
 \leq\,
 \int_{\Omega_\delta^-}((U^h-\varepsilon)_+-(U^h-\varepsilon))e^{2\frac{\phi}{h}}\,(\Psi^{h})^{2}dx\nonumber\\[1.5mm]
\displaystyle &\hspace{2cm} =\int_{\Omega_\delta^-}(U^h-\varepsilon)_-e^{2\frac{\phi}{h}}\,(\Psi^{h})^{2}dx\leq ||U||_{L^\infty}\int_{\Omega_\delta^-}e^{2\frac{\phi}{h}}\,(\Psi^{h})^{2}dx\label{eq.inegsuite}
\end{align}
where we used $\varepsilon^h\leq\varepsilon$ and the inequality
\begin{equation}\label{eq.nablaphi}
\displaystyle |\nabla\phi|^{2}\leq (1-\delta)^2(U^h-\varepsilon)_+, \quad p.p.\Omega
\end{equation}
which follows from \eqref{Agmon distance 1}.\\
On $\Omega_\delta^+$, we have $(U^h-\varepsilon)_+=U^h-\varepsilon$. Using \eqref{eq.nablaphi} again, we obtain that $\forall
 x\in\Omega_\delta^+$
\begin{align*}
\displaystyle U^h+V^{h}-\varepsilon^{h}-|\nabla\phi|^{2}&\geq
 U^h-\varepsilon-|\nabla\phi|^{2} \geq
 U^h-\varepsilon-(1-\delta)^2(U^h-\varepsilon)\\[1.5mm]
 \displaystyle &=\delta(2-\delta)(U^h-\varepsilon)\geq \delta^2
\end{align*}
Injecting the result above in \eqref{eq.inegsuite}, we get:
\[
\displaystyle ||h\nabla(e^{\frac{\phi}{h}}\Psi^{h})||_{L^2(\Omega)}^2+\delta^2\int_{\Omega_\delta^+}e^{2\frac{\phi}{h}}\,(\Psi^{h})^{2}dx\leq ||U||_{L^\infty}\int_{\Omega_\delta^-}e^{2\frac{\phi}{h}}\,(\Psi^{h})^{2}dx
\] 
The domain $\Omega_\delta^-$ is inside the support of $U^h$; therefore,
 in this region, we have $\phi=0$ and taking into account the
 normalization condition $||\Psi^{h}||_{L^2(\Omega)}=1$, it follows:
\[
\displaystyle ||h\nabla(e^{\frac{\phi}{h}}\Psi^{h})||_{L^2(\Omega)}^2+\delta^2\int_{\Omega}e^{2\frac{\phi}{h}}\,(\Psi^{h})^{2}dx\leq (||U||_{L^\infty}+\delta^2)\int_{\Omega_\delta^-}e^{2\frac{\phi}{h}}\,(\Psi^{h})^{2}dx\leq||U||_{L^\infty}+\delta^2
\]  
This gives the estimate \eqref{Agmon estimates 1}.
\end{proof}

\begin{corollary}
\label{Corollary 0}Under the assumptions of Lemma \eqref{Agmon estimates}, the following estimate holds
\begin{equation}
\displaystyle \left\Vert
 he^{c_0\frac{|x-x_0|}{h}}\,\nabla\Psi^h\right\Vert _{L^{2}(\Omega)} +
\left\Vert e^{c_0\frac{|x-x_0|}{h}}\,\Psi^{h}\right\Vert_{L^{2}(\Omega)}\leq C\,, \label{Agmon estimates 2}%
\end{equation}%
for suitable positive constants $c_0$ and $C$.
\end{corollary}

\begin{proof}
As a direct consequence of the estimate (\ref{Agmon estimates 1}) and the
inequality \eqref{compDist}, it follows that the $L^{2}$-norm of the
function $e^{c_0\frac{|x-x_0|}{h}}\,\Psi^{h}$ is
uniformly bounded w.r.t. $h$.

For what concerns the first term in (\ref{Agmon estimates 2}), we
notice that%
\[
\displaystyle e^{c_0\frac{|x-x_0|}{h} -c_{1}}\,\left\vert \nabla\Psi^{h}\right\vert \leq\left\vert e^{\frac{\phi}{h}}\,\nabla\Psi^{h}\right\vert =\left\vert \nabla\left(  e^{\frac{\phi}{h}}\Psi^{h}\right)  -\left(  \nabla e^{\frac{\phi}{h}}\right)  \,\Psi^{h}\right\vert \,.
\]
The term $\nabla e^{\frac{\phi}{h}}$ at the r.h.s. is pointwise bounded by%
\[
\displaystyle \left\vert \nabla e^{\frac{\phi}{h}}\right\vert \leq \frac{1}{h}\left(  U^h-\varepsilon
\right)_{+}^{\frac{1}{2}}e^{\frac{\phi}{h}}\,,
\]
as it comes from (\ref{Agmon distance 1}). Then, using once more the relation
(\ref{Agmon estimates 1}), we obtain%
\[
\displaystyle e^{-c_{1}}\left\Vert
 he^{c_0\frac{|x-x_0|}{h}}\,\nabla\Psi^h\right\Vert _{L^{2}(\Omega)}\leq\left\Vert h\nabla\left(
e^{\frac{\phi}{h}}\Psi^{h}\right)  \right\Vert _{L^{2}(\Omega)}+\sup_{x\in\Omega}\left(  U^h-\varepsilon\right)_{+}^{\frac{1}{2}%
}\,\left\Vert e^{\frac{\phi}{h}}\Psi^{h}\right\Vert _{L^{2}(\Omega)}\leq C\,.
\]
\end{proof}

\end{document}